\documentclass[reqno]{amsart}
\usepackage[utf8]{inputenc}
\usepackage[margin = 1in]{geometry}
\usepackage{latexsym,amsmath,color,amsthm,epsfig,mathrsfs,mathdots,enumerate}
\usepackage{mathtools}
\mathtoolsset{showonlyrefs}
\usepackage{graphicx}
\usepackage{amssymb}
\usepackage{mdframed}
\usepackage{fancyhdr}
\usepackage{tikz-cd}
\usepackage{verbatim}
\usepackage{nicefrac}
\usepackage{bbm}
\usepackage{hyperref}
\hypersetup{
    colorlinks = true,
    citecolor=black,
    filecolor=black,
    linkcolor=black,
    urlcolor=blue
}
\usepackage{amsfonts, esint}
\usepackage{color}
\usepackage{dsfont}

\numberwithin{equation}{section}

 \usepackage{lastpage}
\newtheorem{thm}{Theorem}[section]
\newtheorem{prop}[thm]{Proposition}
\newtheorem{cor}[thm]{Corollary}

\newtheorem{lemma}[thm]{Lemma}
\newtheorem{defn}[thm]{Definition}

\newtheorem{preremark}[thm]{Remark}

\newtheorem*{nota}{Notation}
\newenvironment{remark}{\begin{preremark}\rm}{\medskip \end{preremark}}
\numberwithin{equation}{section}

\newcommand{\norm}[1]{\left\Vert#1\right\Vert}

\newcommand{\R}{\mathbb R}
\newcommand{\eps}{\varepsilon}

\newcommand{\grad} {\nabla}
\newcommand{\lap} {\Delta}

\newcommand{\dd} {\; \mathrm{d}}

\newcommand{\Jap}[1]{\langle #1 \rangle}

\DeclareMathOperator{\tr}{tr}

\definecolor{sh}{RGB}{255,0,100}

\begin{document}

\begin{abstract}
    We establish an a priori estimate for the dissipation of the Fisher information for the space-homogeneous Landau equation with very soft potentials. This work is motivated by the recent breakthrough by Guillen and Silvestre \cite{guillen2023landau}, which proves that the Fisher information is monotone decreasing. As a direct consequence, we show that the Fisher information becomes instantaneously bounded, even if it is not initially bounded. This leads to a proof of the global existence of smooth solutions for the space-homogeneous Landau equation with very soft potentials, given initial data $f_0 \in L^1_{2-\gamma} \cap L \log L$. This result includes the case of the Coulomb potential.
\end{abstract}

\title{Dissipation estimates of the Fisher information for the Landau equation}
\author{Sehyun Ji}
\address[Sehyun Ji]{Department of Mathematics, University of Chicago,  Chicago, Illinois 60637, USA}
\email{jise0624@uchicago.edu}
\maketitle

\section{Introduction}
 We study the global existence of smooth solutions for the space-homogeneous Landau equation of the form of
\begin{equation} \label{eq: Landaueq}
    \partial_t f=q(f)
\end{equation}
where the collision operator $q(f)$ is given by
\begin{equation} \label{eq: Landaucollision}
    q(f)(v)= \partial_{v_i} \int_{\R^3} \alpha(|v-w|) a_{ij}(v-w) (\partial_{v_j}-\partial_{w_j})[f(v)f(w)] \dd w.
\end{equation}
Here the interaction potential $\alpha$ is a given nonnegative function and $a_{ij}(z)=|z|^2 \delta_{ij}-z_i z_j$. We are mostly concerned with the case $\alpha(r)=r^\gamma$ with $\gamma \in[-3,-2)$, which is referred as very soft potentials. In particular, this range includes the most important case $\gamma=-3$, the Coulomb potential.
A recent breakthrough by Guillen and Silvestre \cite{guillen2023landau} proves the global existence of smooth solutions to \eqref{eq: Landaueq} for very soft and Coulomb potentials, given $C^1$ initial data with Maxwellian upper bounds. The key insight of their proof is proving the monotonicity of the Fisher information, $i(f)$, which is defined as follows.
\begin{defn}[The Fisher information]
For a function $f:\R^3 \to \R{\ge 0}$, 
\begin{equation}
    i(f):=\int_{\R^3} f(v) |\grad \log f(v)|^2 \dd v=\int_{\R^3} \frac{| \grad f(v)|^2}{f(v)} \dd v =4\int_{\R^3} |\grad \sqrt{f(v)}|^2 \dd v.
\end{equation}
\end{defn}
\noindent
It is well known that a solution to the Landau equation \eqref{eq: Landaueq} conserves the mass, momentum, and energy: 
 \begin{equation}
       (mass):= \int_{\R^3}  f(v) \dd v,
    \end{equation}
    \begin{equation}
        (momentum):=\int_{\R^3} f(v) v \dd v,
    \end{equation}
    \begin{equation}
        (energy):=\int_{\R^3} f(v) |v|^2 \dd v.
    \end{equation}  
Another notable feature of a solution is that its entropy decreases in time:
     \begin{equation}
        (entropy):=h(f)=\int_{\R^3} f(v) \log f(v) \dd v.
    \end{equation}
Throughout the paper, we consider an initial data $f_0$ with mass $M_0$, energy $E_0$, and entropy $H_0$. Then, at any time $t$, the solution $f(t)$ will have mass $M_0$, energy $E_0$, and entropy bounded above by $H_0$.\\

 In this paper, we derive an a priori estimate for the dissipation of the Fisher information, $-\partial_t i(f)$, and use an elementary ODE argument to prove that the Fisher information decreases at a rate of $\frac{1}{t}$ near $t=0$.
 Our main result is the following.

\begin{thm} \label{thm: fisherdissipation}
Suppose $\alpha(r)=r^\gamma$ for $\gamma \in[-3,-2)$. Let $f:[0,T]\times \R^3$ be a solution of the Landau equation \eqref{eq: Landaueq} with initial data $f_0$ that has mass $M_0$, energy $E_0$, and entropy $H_0$. Then,  
    \begin{equation}
    -\partial_t i(f) \ge   c_1\int_{\R^3} \Jap{v}^{\gamma-2} f(v) \norm{\grad^2 \log f(v)}^2  \dd v  - C_1 M_0 \cdot i(f) - C_2 M_0^2.
\end{equation}
Here $C_1,C_2$ are absolute constants and $c_1$ depends on $\gamma, M_0,E_0,H_0$.
\end{thm}

In \cite{desvillettes2015entropy}, Desvillettes obtained a lower bound for the entropy dissipation in terms of a weighted Fisher information. The weight in the entropy dissipation estimate was later improved by the author \cite{ji2024entropy}. Theorem \ref{thm: fisherdissipation} could be interpreted as a higher order version of that estimate, since we obtained a lower bound of the dissipation of Fisher information in terms of a second order coercive functional.

\begin{remark}
    With minor modifications, Theorem \ref{thm: fisherdissipation} can be extended to  other values of $\gamma$, such as moderately soft and hard potentials. However, the regularity results \cite{ gualdani2019apweights, silvestre2017upperboundLandau, wu2014soft} that are sufficient to establish the global existence of smooth solutions are well known, so we focus on the case of very soft potentials.
\end{remark}
\noindent
If we further impose a bound
\begin{equation} \label{eq: highermomentbound}
    \int_{\R^3} f_0(v) |v|^{2-\gamma} \dd v \le W_0,
\end{equation}
we can derive the following inequality using the Cauhcy-Schwarz inequality:
\begin{equation}
    -\partial_t i(f) \ge   c_2 i(f)^2  +\text{(lower order terms)}
\end{equation}
where $c_2$ only depending on $T,\gamma, M_0,W_0,H_0$. See Corollary \ref{cor: fisherdissipation} and Section 3 for the precise statement and detailed proof. 
By solving the ODE directly, we obtain the following theorem. 
\begin{thm}\label{thm: Fisherdecay}
    Suppose $\alpha(r)=r^\gamma$ for $\gamma \in [-3,-2)$. Let $f:[0,T]\times \R^3 \to \R_{\ge 0}$ be a solution of the Landau equation \eqref{eq: Landaueq} with initial data $f_0$ that has mass $M_0$, energy $E_0$, and entropy $H_0$. Furthermore, assume the bound \eqref{eq: highermomentbound}.
    Then, for every $t \in (0,T]$, a priori estimate for the Fisher information
    \begin{equation}
        i(f(t)) \le C_0 \left( 1+\frac{1}{t} \right)
    \end{equation}
    holds for some $C_0>0$ only depending on $\gamma, T, M_0, W_0, H_0$.
\end{thm}
A significant corollary of Theorem \ref{thm: Fisherdecay} is that the Fisher information $i(f)$ becomes bounded instantaneously even if it is initially unbounded. It follows that the $L^3$ norm of $f$ becomes bounded instantaneously as well, thanks to the Sobolev embedding $L^6(\R^3) \hookrightarrow \dot{H}^1(\R^3)$.
As a consequence, we  prove the global existence of smooth solutions to the Landau equation \eqref{eq: Landaueq} with very soft potentials, which includes the Coulomb potential. 

\begin{thm} \label{thm: globalexistence}
    Suppose $\alpha(r)=r^\gamma$ for $\gamma\in[-3,-2)$. Consider an initial data $f_0$ with mass $M_0$, energy $E_0$, and entropy $H_0$. Furthermore, assume the bound \eqref{eq: highermomentbound}.
    Then, there exists a smooth global solution $f:[0,\infty)\times\R^3 \to \R_{\ge 0}$ to the Landau equation \eqref{eq: Landaueq}, which is strictly positive for any positive time $t$. The initial data $f_0$ is achieved in the following sense: for every $\psi \in C^2_c(\R^3)$, the map $t\mapsto \int_{\R^3} f(t) \psi \dd v$ is H\"older continuous of order $\frac12$. 
\end{thm}

In 1998, C\'edric Villani \cite{villani1998newclass} constructed a global-in-time weak solution for the Landau equation with general initial data in $L^1_2 \cap L\log L$, known as an $H$-solution. However, the $H$-solution was defined in such a weak sense that it was even not immediately clear whether it satisfied the weak formulation in the classical sense. Later, it was proved by Desvillettes \cite{desvillettes2015entropy} that the $H$-solution is indeed the weak solution in the classical sense. Recently, Guillen and Silvestre \cite{guillen2023landau} constructed global smooth solutions but their analysis only considered $C^1-$smooth initial data.
Building upon these results, Golding, Gualdani, and Loher \cite{golding2024globalsmoothsolutionslandaucoulomb} build global smooth solutions, assuming initial data belongs to the $L^{3/2}_k$ space. Their argument relies on the local-in-time propagation of the $L^{3/2}$ norm. 
In Theorem \ref{thm: globalexistence}, we extend these results by constructing global smooth solutions to the Landau equation \eqref{eq: Landaueq} under the same minimal assumptions as in Villani's original work, except for a higher $L^1$ moment condition.

In \cite{fournier2009local}, Fournier and Gu\'erin proved the uniqueness of the solution to the Landau equation for $\gamma \in (-3,0)$ if the solution belongs to $L^1_t L^{3/(3+\gamma)}_v(\R^3)$. Fournier \cite{fournier2010uniqueness} also proved the uniqueness of the solution for the case of the Coulomb potential if the solution belongs to $L^1_t L^\infty_v(\R^3)$. In addition, Chern and Gualdani \cite{chern2022uniqueness} established the uniqueness for solutions in $L^\infty(0,T,L^p)$ for $p>3/2$.
Our estimates do not directly imply these assumptions, therefore the uniqueness of the solution is not guaranteed; there might be a bifurcation of the solution at time $t=0$. However, since the solutions become instantaneously smooth, once the value of the solution is determined for a small interval of time, there will be a unique continuation from that point on.

\begin{nota}
    We use the Japanese bracket notation $\Jap{v}:=\sqrt{1+|v|^2}$. The weighted Lebesgue norm $L^p_m$ of $f(v)$ is equal to the $L^p$ norm of $\Jap{v}^m f(v)$.
\end{nota}
\noindent
Using the above notation, the conditions on initial data $f_0$ in Theorem \ref{thm: globalexistence} can be also written as 
\begin{equation}
    f_0 \in L^1_{2-\gamma} \cap L \log L.
\end{equation}
\begin{remark} \label{remark: LlogL}
    It is well known that a function with finite mass, energy, and entropy belongs to $L\log L$. In other words,
    \begin{equation} \label{eq: LlogL}
        \int_{\R^3} f(v) |\log f(v) | \dd v \le C(M_0,E_0,H_0)
    \end{equation}
    holds for some $C(M_0,E_0,H_0)$.
    We provide a short proof in Appendix \ref{appendix : 1} for reader's convenience.
\end{remark}

Very recently, the monotonicity of the Fisher information for the Boltzmann equation is proved by Imbert, Silvestre, and Villani \cite{imbert2024monotonicityfisherinformationboltzmann}. Given the similarity in structure between the Landau and the Boltzmann equations, we can expect an analogous result of Theorem \ref{thm: fisherdissipation} or \ref{thm: Fisherdecay}; however this remains to be addressed in future work.

\section{Preliminaries}
In this section, we review the setup and results from the proof that Fisher information is monotone decreasing, as presented in \cite{guillen2023landau} by Guillen and Silvestre. Throughout, we will mostly adopt the notations used in \cite{guillen2023landau} to maintain consistency with their formulation. 
\subsection{Lifting}
Given a function $f(v)$ on $\R^3$, we can define a lifted function $F(v,w)$ on $\R^6$ as
\begin{equation}
    F(v,w):=(f\otimes f)(v,w)=f(v)f(w).
\end{equation}
Conversely, we can define a projection operator $\pi:L^1(\R^6) \to L^1(\R^3)$ as
\begin{equation}
    \pi F(v):= \int_{\R^3} F(v,w) \dd w.
\end{equation}
Note that if $F$ is a probability distribution, then $\pi F$ is  a marginal distribution of $F$. The key observation is interpreting the Landau equation \eqref{eq: Landaueq} in terms of $F(v,w)$. 
For a function $F(v,w):\R^6 \to \R_{\ge 0}$, define the degenerate elliptic operator
\begin{equation} \label{defn: Qcollision}
    Q(F)(v,w):= (\partial_{v_i} -\partial_{w_i}) [\alpha(|v-w|)a_{ij}(v-w) (\partial_{v_j}-\partial_{w_j})F].
\end{equation}
For any given $f$ on $\R^3$ with mass $1$, consider the following initial value problem
\begin{equation} \label{eq: initialvalue}
    \begin{cases}
        \partial_t F=Q(F),\\
        F(0,v,w)=f(v)f(w).
    \end{cases}
\end{equation}
Note that the function $F(t)$ stays symmetric due to the symmetry of the operator $Q(F)$.
Then, the projected function $\pi F$ solves the Landau equation \eqref{eq: Landaueq} evaluated at $t=0$. That is,
\begin{equation}
        \begin{cases}
        \partial_t (\pi F)\vert_{t=0} =q(\pi F)\vert_{t=0}=q (f),\\
        \pi F(0,v)=f(v).
        \end{cases}
\end{equation}
Now, we relate the Fisher information of $f$ and $F$. We use a notation $I$ for the Fisher information of a function in the lifted variables,
\begin{equation}
    I(F):=\iint_{\R^6} F |\grad \log F|^2 \dd w \dd v.
\end{equation}
Guillen and Silvestre prove the following lemma using the projection inequality
\begin{equation} \label{ineq: lifting}
    i(\pi F) \le \frac{1}{2} I(F),
\end{equation}
where $F$ is symmetric, and the equality holds if $F(v,w)= f(v)f(w)$ for some $f$.
\begin{lemma}[Lemma 3.4 and Remark 3.5 from \cite{guillen2023landau}] \label{lem: liftedFisher}
Let $f$ be a solution of \eqref{eq: Landaueq} and $F$ be a solution of \eqref{eq: initialvalue}. Then, at $t=0$,
    \begin{equation} 
        \partial_t i(f)=\frac12 \partial_t I(F).
    \end{equation}
\end{lemma}

\begin{remark}
    Lemma \ref{lem: liftedFisher} remains valid even without normalizing the mass.
    If the mass of $f$ is $M_0$, then we have the relation:
    \begin{equation}
        \frac{1}{2}I(f\otimes f)= M_0 i(f) =i(M_0 f).
    \end{equation}
    However, the projection of $f\otimes f$ is no longer $f$, but $M_0 f$. These factors $M_0$ cancel out, so Lemma \ref{lem: liftedFisher} still holds, independent of the normalization of mass.
\end{remark}

Studying the flow \eqref{eq: initialvalue} is significantly easier than studying the original Landau equation \eqref{eq: Landaueq}, as $Q(F)$ is a classical second-order operator with local coefficients and, most importantly, is linear in $F$.

\subsection{The formula for the dissipation of the Fisher information}
We decompose the matrix $a_{ij}$
\begin{equation}
    a(v-w)=\sum_{k=1}^3 b_k(v-w) \otimes b_k(v-w)^t,
\end{equation}
where the vector field $b_k$ is defined as
    \begin{equation}
        b_k(v-w):= e_k \times  (v-w),
    \end{equation}
with $e_k$ denoting the standard basis of $\R^3$.    
Note that $b_k$ is perpendicular to $v-w$. We define the lifted vector field $\tilde{b}_k$ of $b_k$.
\begin{equation}
        \tilde{b}_k(v-w):=
        \begin{pmatrix}
            b_k(v-w)\\
            -b_k(v-w)
        \end{pmatrix}
        .
\end{equation}
Then, the operator $Q(F)$ can be written as
\begin{equation}
    Q(F)= \sum_{k=1}^3  \tilde{b}_k\cdot \grad ( \alpha \tilde{b}_k \cdot \grad F).
\end{equation}
We state the formula for the dissipation of the Fisher information from \cite{guillen2023landau}.
\begin{prop}[Lemma 8.2 from \cite{guillen2023landau}] \label{prop: Fisher}
    The dissipation of the Fisher information for the initial value problem \eqref{eq: initialvalue} is
    \begin{align}
        \frac12 \partial_t I(F)
        &=\frac12 \Jap{I'(F),Q(F)}\\
        &= \sum_{k=1}^3  \iint_{\R^6} -F|\grad[\sqrt{\alpha}\tilde{b}_k \cdot \grad \log F] |^2 \dd w \dd v+\iint_{\R^6} \frac{\alpha'(|v-w|^2}{2\alpha(|v-w|)} F\left(\tilde{b}_k\cdot \grad \log F \right)^2 \dd w \dd v.
    \end{align}
\end{prop}
\noindent
Verifying that the expression from the above proposition is non-positive can be reduced to the functional inequality involving the derivatives of logarithm of a function on $S^2$: Bakry-\'Emery $\Gamma_2$ criterion.

\subsection{Bakry-\'Emery $\Gamma_2$ criterion}
We denote by $\Lambda_3$ the largest constant such that  
\begin{equation} \label{eq: Bakryemery}
     \int_{S^2} f \Gamma_2 (\log f, \log f) \dd \sigma \ge \Lambda_3 \int_{S^2} f |\grad_\sigma \log f|^2 \dd \sigma
\end{equation}
holds for every positive symmetric function $f$ on $S^2$.
Here $\grad_\sigma$ means the spherical derivative at $\sigma \in S^2$ and $\Gamma_2(g,g) = \norm{\grad_\sigma^2 g}^2 +|\grad_\sigma g|^2$ for a function $g$ on $S^2$. In the literature, this functional inequality is called as the Bakry-\'Emery $\Gamma_2$ criterion inequality \eqref{eq: Bakryemery}, which was introduced as a criterion for the logarithmic Sobolev inequality.
Guillen and Silvestre proved that 
\begin{equation}
     -\frac12 \Jap{I'(F),Q(F)} \ge 0
\end{equation}
if the interaction potential $\alpha$ satisfies
     \begin{equation} \label{eq: conditionalpha}
        \sup_{r>0}  \frac{r|\alpha'(r)|}{\alpha(r)} \le \sqrt{4\Lambda_3}.
    \end{equation}
 For $\alpha(r)=r^\gamma$, the left hand side of \eqref{eq: conditionalpha} is equal to $|\gamma|$. A classical result by Bakry-\'Emery \cite{bakry1985diffusion} shows that $\Lambda_3 \ge 2$ without the symmetry condition, but this does not apply to the case of the Coulomb potential.
In \cite{guillen2023landau}, Guillen and Silvestre showed that $\Lambda_3 \ge 4.75$, which is sufficient to cover the Coulomb potential case. Recently, the author improved the lower bound to $\Lambda_3 \ge 5.5$ in \cite{ji2024bakryemery}. We present these results as theorems below.
\begin{thm}[Theorem 1.1 from \cite{ji2024bakryemery}]\label{thm: BakryEmery}
    The largest constant $\Lambda_3$ for the inequality \eqref{eq: Bakryemery} is greater than or equal to $5.5$.
\end{thm}

\begin{thm}[Theorem 1.1 from \cite{guillen2023landau}, Corollary 1.3 from \cite{ji2024bakryemery}] \label{thm: Fisherdecreaserange}
    Let $f:[0,T]\times \R^3 \to \R_{\ge 0}$ be a solution of the Landau equation \eqref{eq: Landaueq}. 
    If the interaction potential $\alpha$ satisfies
    \begin{equation}
        \sup_{r>0}  \frac{r|\alpha'(r)|}{\alpha(r)} \le \sqrt{22},
    \end{equation}
    then the Fisher information $i(f)$ is decreasing in time.
    In particular, the Fisher information is monotone decreasing in time for very soft and Coulomb potentials.
\end{thm}
See \cite{bakry2014book,ji2024bakryemery} and the references therein for the relationship between the inequality \eqref{eq: Bakryemery} and the logarithmic Sobolev inequality, as well as for further discussion. See Section 3 for a more detailed explanation of the proof of Theorem \ref{thm: Fisherdecreaserange}.
From Theorem \ref{thm: Fisherdecreaserange}, Guillen and Silvestre prove the global existence of smooth solutions for $C^1$ initial data with Maxwellian upper bounds.
\begin{thm}[Theorem 1.2 from \cite{guillen2023landau}] \label{thm: luisglobal}
    Suppose $\alpha(r)=r^\gamma$ for $\gamma\in [-3,-2)$. Let $f_0: \R^3 \to \R_{\ge 0}$ be an initial data that is $C^1$ and is bounded by a Maxwellian in the sense that
    \begin{equation}
        f_0(v) \le C_0 \exp(-\beta |v|^2),
    \end{equation}
    for some positive parameters $C_0$ and $\beta$. Assume also that $i(f_0) <+\infty$. Then, there is a unique global classical solution $f:[0,\infty)\times \R^3 \to \R_{\ge 0}$ to the Landau equation \eqref{eq: Landaueq}, with initial data $f(0,v)=f_0(v)$. For any $t>0$, this function $f$ is strictly positive, in the Schwartz space, and bounded above by a Maxwellian. The Fisher information $i(f)$ is non-increasing in time.
\end{thm}
\subsection{Known results}
We review results from the literature that were known prior to the breakthrough by Guillen and Silvestre. These results\textemdash such as  conditional regularity results, the propagation of higher $L^1$ moments, and entropy dissipation estimates\textemdash are essential for proving the global existence of smooth solutions. Throughout this section, we focus on the case $\alpha(r)=r^\gamma$ for $\gamma \in [-3,-2)$.

We start with a concentration-type lemma from \cite{silvestre2016newregularization} or \cite{silvestre2017upperboundLandau}. The heuristic is that bounded energy prevents mass from escaping to the infinity, while bounded entropy prevents mass from accumulating on a small set.
\begin{lemma}[Lemma 4.6 from \cite{silvestre2016newregularization} or Lemma 3.3 from \cite{silvestre2017upperboundLandau}] \label{lem: concentration}
    Suppose that $f$ has mass $M_0$, energy $E_0$ and entropy bounded above by $H_0$. Then, there exist positive numbers $R,\mu,l$ depending on $M_0,E_0,H_0$ such that
    \begin{equation}
        | \{ v\in B_R(0):  f(v) \ge l \} | \ge \mu.
    \end{equation}   
\end{lemma}
\noindent
From Lemma \ref{lem: concentration}, we can derive the lower bound of the diffusion matrix $A[f]:=f\ast (|\cdot|^\gamma a_{ij})$.
\begin{lemma}[Lemma 3.2 from \cite{silvestre2017upperboundLandau}] \label{lem: lowerdiffusion}
   Under the same assumptions in Lemma \ref{lem: concentration},
     \begin{equation}
        A[f](v) \ge c_0 \left( \Jap{v}^{\gamma} \left(\frac{v}{|v|}\right)^{\otimes 2} +\Jap{v}^{\gamma+2} \left[I_3-\left(\frac{v}{|v|}\right)^{\otimes 2} \right]\right)
    \end{equation}
    for every nonzero $v\in \R^3$,
    where $c_0$ depending on $M_0, E_0, H_0$.
    In particular,
    \begin{equation}
        A[f](v) \ge c_0 \Jap{v}^{\gamma} I_3.
    \end{equation}
\end{lemma}

Next, we review estimates of the upper bounds for the diffusion matrix $A[f]$, citing the following results from \cite{bowman2024isotropiclandau} and \cite{silvestre2017upperboundLandau}: 

\begin{lemma}[Lemma A.3 from \cite{bowman2024isotropiclandau}] \label{lem: improveweight}
    Let $1 \le p \le \infty$. Suppose $f \in L^p_k(\R^3) $ for $k>3-\frac{3}{p}$. Then, for $\frac3p-3<\mu<0$,
    \begin{equation}
        \norm{f\ast |\cdot|^\mu}_{L^\infty(\R^3)} \le C(\mu, p,k) \Jap{v}^{\mu} \norm{f}_{L^p_k(\R^3)}.
    \end{equation}
    for some $C(\mu, p,k)>0$.
\end{lemma}
\begin{lemma}[From \cite{silvestre2017upperboundLandau}] \label{lem: interpolation2}
    Suppose $f\in L^1(\R^3) \cap L^\infty(\R^3)$. Then, for $-3<\mu<0$,
    \begin{equation}
        \norm{f\ast |\cdot|^\mu}_{L^\infty(\R^3)} \le C(\mu) \norm{f}_{L^1(\R^3)}^{1+\frac{\mu}{3}} \norm{f}_{L^\infty(\R^3)}^{-\frac{\mu}{3}}
    \end{equation}
    for some $C(\mu)>0$.
\end{lemma}
\noindent
Note that Lemma \ref{lem: improveweight} is an improved version of the estimate in \cite{silvestre2017upperboundLandau}, resulting in a slight improvement in the weight for the conditional regularity result.
There are several conditional regularity results available in the literature. See 
 \cite{alonso2024prodiserrin, golding2024local, gualdani2019apweights, silvestre2017upperboundLandau}.
 Among those results, we cite the conditional regularity result from \cite{silvestre2017upperboundLandau}. 
By applying Lemmas \ref{lem: lowerdiffusion} and \ref{lem: improveweight} to the result of \cite{silvestre2017upperboundLandau}, we derive the following theorem.
\begin{thm}[Improved version of Theorem 3.8 from \cite{silvestre2017upperboundLandau}] \label{thm: conditional}
    Let $f:[0,T]\times \R^3 \to \R_{\ge 0}$ be a solution of the equation \eqref{eq: Landaueq} with initial data $f_0$ that has mass $M_0,$ energy $E_0,$ and entropy $H_0$. If there exist some $p \in (\frac{3}{5+\gamma},\frac{3}{3+\gamma})$ and $k >\frac1p \max\left(3p-3,-\frac{3\gamma}{2}-1 \right)$ such that
    \begin{equation}
        \norm{f(t)}_{L^p_k(\R^3)} \le K_0
    \end{equation}    
     for every $t\in[0,T]$, then
    \begin{equation}
        \norm{f(t)}_{L^\infty(\R^3)} \le C \left(1+ t^{-\frac{3}{2p}} \right)
    \end{equation}
    holds for every $t\in (0,T]$ where $C$ only depends on $\gamma, p,k, M_0,E_0,H_0, K_0$.
\end{thm}
\begin{proof}[Sketch of proof]
Since $p> \frac{3}{5+\gamma}$ and $k>3-\frac3p$, Lemma \ref{lem: improveweight} implies that
\begin{equation} \label{eq: upperbound}
    |A[f](v)| \le \norm{f\ast |\cdot|^{\gamma+2}}_{L^\infty(\R^3)} \le C(\gamma,p,k) \Jap{v}^{\gamma+2} \norm{f}_{L^p_k(\R^3)}.
\end{equation}
\end{proof}

\noindent
We can rewrite the equation \eqref{eq: Landaueq} into
\begin{equation}
    \partial_t f =\grad \cdot (A[f]f-b[f]f)
\end{equation}
where
\begin{equation}
    A[f](v)=\int_{\R^3}f(v-w)|w|^\gamma a_{ij}(w)\dd w, \quad b[f](v):=-2\int_{\R^3} f(v-w) |w|^\gamma w \dd w .
\end{equation}
From Lemma \ref{lem: interpolation2}, we obtain
    \begin{equation} \label{eq: upperdiffusion}
    \norm{A[f]}_{L^\infty(\R^3)} \le \norm{f*|\cdot|^{\gamma+2}}_{L^\infty(\R^3)} \le C(\gamma) \norm{f}^{1+\frac{\gamma+2}{3}}_{L^1(\R^3)} \norm{f}_{L^\infty(\R^3)}^{\frac{-\gamma-2}{3}}
\end{equation}
and
\begin{equation} \label{eq: upperfirstorder}
    \norm{b[f]}_{L^\infty(\R^3)} \le \norm{f*|\cdot|^{\gamma+1}}_{L^\infty(\R^3)} \le C(\gamma) \norm{f}^{1+\frac{\gamma+1}{3}}_{L^1(\R^3)} \norm{f}_{L^\infty(\R^3)}^{\frac{-\gamma-1}{3}}
\end{equation}
for some $C(\gamma)>0$.
Note that \eqref{eq: upperdiffusion}, \eqref{eq: upperfirstorder}, and Lemma \ref{lem: lowerdiffusion} show that \eqref{eq: Landaueq} is a uniformly parabolic equation locally, given the $L^\infty$ bound of $f$. This interpretation will be used to obtain higher regularity estimates.

Next, a result regarding the propagation of higher $L^1$ moments can be found in \cite{carrapatoso2017large} or \cite{desvillettes2015entropy}: 
\begin{thm}[Lemma 8 from \cite{carrapatoso2017large}]\label{thm: highermoment}
    Let $f:[0,T]\times \R^3 \to \R_{\ge 0}$ be a solution of the  equation \eqref{eq: Landaueq} with initial data $f_0$ that has mass $M_0$, energy $E_0$, and entropy $H_0$. If $f_0$ belongs to  $L^1_\kappa(\R^3)$ for some $\kappa > 2$, then
    \begin{equation}
        \norm{f(t)}_{L^1_\kappa(\R^3)} \le C(\norm{f_0}_{L^1_\kappa(\R^3)}+t)
    \end{equation}
    for every $t\in [0,T]$ where $C$ only depends on $\gamma, \kappa, M_0, E_0, H_0$.
\end{thm}

Lastly, we review the entropy dissipation estimates. By symmetrizing, it is straightforward to verify that
\begin{equation} \label{eq: entropydissipation}
    -\partial_t h(f) = \frac12 \iint_{\R^6} f(v)f(w)|v-w|^{\gamma} a(v-w) \left(\frac{\grad f(v)}{f(v)}-\frac{\grad f(w)}{f(w)}\right) \left(\frac{\grad f(v)}{f(v)}-\frac{\grad f(w)}{f(w)}\right)  \dd w \dd v \ge 0.
\end{equation}
The entropy dissipation structure \eqref{eq: entropydissipation} was crucial in Villani's construction of the $H$-solution.
More than being just non-negative, coercitivity estimates have been studied for the entropy dissipation. See \cite{alexandre2000entropydissipation,desvillettes2015entropy, ji2024entropy, mouhot2006coercivity}.
In the case of the Coulomb potential, Desvillettes proved that the entropy dissipation belongs to $L^3_{-3}(\R^3)$ in \cite{desvillettes2015entropy}. He used this estimate to prove that the $H$-solution is, in fact, the classical weak solution. Later, the author improved the weight to $L^3_{-5/3}(\R^3)$ in \cite{ji2024entropy} and showed that this weight is optimal. These results can be easily extended to very soft potentials, leading to the conclusion that the entropy dissipation belongs to $L^3_{(3\gamma+4)/3}(\R^3)$.

\section{A priori estimate for the dissipation of the Fisher information}
In this section, we show an a priori estimate for the dissipation of the Fisher information.
We begin by proving Theorem \ref{thm: Fisherdecay}, assuming the conclusion of Theorem \ref{thm: fisherdissipation}. First, 
we derive a simple ODE-type inequality through a straightforward argument. 

\begin{cor}\label{cor: fisherdissipation}
    Suppose $\alpha(r)=r^\gamma$ and $\gamma \in[-3,-2)$. Let $f:[0,T]\times \R^3$ be a solution of the Landau equation \eqref{eq: Landaueq} with initial data $f_0$ that has mass $M_0$, energy $E_0$, and entropy $H_0$. Furthermore, assume 
    \begin{equation}
        \int_{\R^3} f_0(v) |v|^{2-\gamma } \dd v\le W_0.
    \end{equation}
    Then,  
    \begin{equation}
    -\partial_t i(f) \ge   c_2 i(f)^2  - C_1 M_0 \cdot i(f) - C_2 M_0^2.
\end{equation}
Here $C_1,C_2$ are absolute constants and $c_2$ depends on $T, \gamma, M_0,W_0,H_0$.
\end{cor}

\begin{remark}
Note that a dependence on $E_0$ is unnecessary because it can be controlled by $M_0$ and $W_0$.
\end{remark}

\begin{proof}
By the Cauhcy-Schwarz inequality,
\begin{equation}
    \norm{\grad^2 \log f(v)}^2 \ge \sum_{i=1}^3 (\partial_{v_i v_i} \log f(v))^2 \ge \frac13 (\lap \log f(v))^2.
\end{equation}
It follows that
\begin{equation}
    \int_{\R^3} \Jap{v}^{\gamma-2} f(v) \norm{\grad^2 \log f(v)}^2 \dd v \ge \frac13 \int_{\R^3} \Jap{v}^{\gamma-2} f(v)(\lap \log f(v))^2  \dd v.
\end{equation}
Applying the Cauchy-Schwarz inequality again,
\begin{equation}
   \left( \int_{\R^3} \Jap{v}^{\gamma-2} f(v) (\lap \log f(v))^2 \dd v \right) \left(\int_{\R^3}f(v)  \Jap{v}^{2-\gamma} \dd v \right) \ge \left( \int_{\R^3} f(v) \lap \log f(v) \dd v \right)^2 = i(f)^2.
\end{equation}
Recall that Theorem \ref{thm: highermoment} says every higher $L^1$ moment of $f$ has at most linear growth in time, thus,
\begin{equation}
    \int_{\R^3} f(v)  \Jap{v}^{2-\gamma} \dd v \le C
\end{equation}
for some $C=C(T,\gamma, M_0, H_0, W_0)$, which concludes the proof.
\end{proof}
\begin{proof} [Proof of Theorem \ref{thm: Fisherdecay}]
    From Corollary \ref{cor: fisherdissipation}, the inequality
    \begin{equation}
        -\partial_t i(f) \ge c_2 i(f)^2- C_1 M_0 \cdot  i(f) -C_2 M_0^2 
    \end{equation}
holds. It is clear that there exists $\tilde{C}>0$ depending on $c_2,C_1,C_2$ such that
\begin{equation}
    -\partial_t i(f) > \frac{c_2}{2} i(f)^2 -\tilde{C} M_0^2
\end{equation}
by the AM-GM inequality.\\
Let's prove by contradiction. Suppose that
\begin{equation}
    i(f(t)) \ge \sqrt{\frac{4\tilde{C}}{c_2}}M_0 +\frac{4}{c_2 t }
\end{equation}
holds.  It follows that
\begin{equation}
    -\partial_t i(f(t)) \ge \frac{c_2}{2} i(f(t))^2 -\tilde{C} M_0^2 >  \frac{c_2}{4} i(f(t))^2
\end{equation}
Solving the ODE, we obtain
\begin{equation}
    \frac{1}{i(f(t))} > \frac{c_2 t}{4},
\end{equation}
which is a contradiction. This concludes the proof of Theorem \ref{thm: Fisherdecay}.    
\end{proof}
It remains to prove Theorem \ref{thm: fisherdissipation}.
First, we explain in detail the idea of the proof that the Fisher information is decreasing.
Recall that
\begin{align}\label{eq: Fisher}
    -\frac{1}{2} \partial_t I(F)
    &=\iint F |\grad (\sqrt{\alpha} \tilde{b}_k\cdot \grad\log F)|^2 \dd v \dd w - \iint F \frac{\alpha'^2}{2\alpha} |\tilde{b}_k\cdot \grad\log F|^2 \dd v \dd w.
\end{align}
from Proposition \ref{prop: Fisher}. We further define normal vectors $n_0 \in \R^3$ and $n \in \R^6$ as follows:
 \begin{equation}
        n_0=\frac{v-w}{|v-w|}, \quad   n=\frac{1}{\sqrt{2}}
        \begin{pmatrix}
            n_0\\
            -n_0
        \end{pmatrix}
        .
    \end{equation}
Then, the matrix $a_{ij}$ can be also written as
\begin{equation}
    a(v-w)=|v-w|^2 I_3 -|v-w|^2 n_0 n_0^t.
\end{equation}

\begin{lemma}[Lemma 8.3 from \cite{guillen2023landau}]\label{lem: decomposegradient}
Given any differentiable function $G:\R^6 \to \R$, the following identity holds at every point $(v,w)\in \R^6$.
\begin{equation} 
    |\grad G|^2=\frac12 \sum_{i=1}^3 |(\grad_{v_i}+\grad_{w_i})G|^2+|n\cdot \grad G|^2 +\frac{1}{2|v-w|^2}\sum_{k=1}^3|\tilde{b}_k \cdot \grad G|^2.
\end{equation}
\end{lemma}
 We associate quadratic forms in $6$ variables to $6\times 6$ matrices. For example, we associate $|\grad G|^2$ to the identity matrix $I_6$ since
\begin{equation}
    |\grad G|^2 = (\grad G)^t  \begin{pmatrix}
\begin{matrix} I_3 \\  \end{matrix} &0 \\ 0 & I_3 \\ 
\end{pmatrix} \grad G.
\end{equation}
Hence, we can interpret Lemma \ref{lem: decomposegradient} as the decomposition of $I_6$ into matrices:
\begin{equation}
    \begin{pmatrix}
\begin{matrix} I_3 \\  \end{matrix} &0 \\ 0 & I_3 \\ 
\end{pmatrix}=\frac{1}{2}\begin{pmatrix}
\begin{matrix} I_3 \\  \end{matrix} & I_3 \\ I_3 &  I_3 \\  
\end{pmatrix}
+\frac{1}{2}\begin{pmatrix}
\begin{matrix} n_0n_0^t \\  \end{matrix} & -n_0 n_0^t \\ -n_0 n_0^t&  n_0 n_0^t   \\  
\end{pmatrix}
+
\frac{1}{2|v-w|^2}\begin{pmatrix}
\begin{matrix} a_{ij}(v-w) \\  \end{matrix} & -a_{ij}(v-w) \\ -a_{ij}(v-w) &  a_{ij}(v-w) \\ 
\end{pmatrix}.
\end{equation}
This decomposition leads to
    \begin{align}  \label{eq: decomposition}  
    -\frac12 \Jap{I'(F),Q(F)}
    &=D_{par}+D_{rad}+D_{sph}-\sum_{k=1}^3 \iint_{\R^6} \frac{(\alpha')^2}{2\alpha} F |\tilde{b}_k\cdot \grad \log F|^2 \dd w \dd v
    \end{align}
where
\begin{equation}
    D_{par}:=\frac12 \sum_{i,j=1,2,3}\iint_{\R^6} \alpha(|v-w|) F|(\partial_{v_i}+\partial_{w_i})\tilde{b}_j\cdot \grad \log F|^2 \dd w \dd v,
\end{equation}
\begin{equation}
    D_{rad}:=\sum_{i=1}^3 \iint_{\R^6} F|n\cdot \grad \left(\sqrt{\alpha} \tilde{b}_i\cdot \grad \log F \right)|^2 \dd w \dd v,
\end{equation}
\begin{equation}
    D_{sph}:=\sum_{i,j=1,2,3} \iint_{\R^6} \frac{\alpha}{2|v-w|^2} F |\tilde{b}_i \cdot \grad (\tilde{b}_j \cdot \grad \log F)|^2 \dd w \dd v.
\end{equation}
Note that $\alpha$ can be factored out from the parentheses in the expressions for $D_{par}$ and $D_{sph}$ because $(\partial_{v_i}+\partial_{w_i})\sqrt{\alpha}=0$ and $\tilde{b}_i \cdot \grad{\sqrt{\alpha}}=0$.
We also define
\begin{equation}
    R_{sph}:=\sum_{k=1}^3 \iint_{\R^6} \frac{\alpha}{|v-w|^2} F|\tilde{b}_k \cdot \grad \log F|^2 \dd w \dd v.
\end{equation}

Note that
\begin{equation}
    \sum_{i,j=1,2,3} \int_{S^2}f(b_i \cdot \grad (b_j \cdot \grad \log f))^2 \dd \sigma = \int_{S^2} f \left( \norm{\grad_\sigma^2 \log f}^2 +|\grad_\sigma \log f|^2 \right) \dd \sigma=\int_{S^2} f \Gamma_2 (\log f, \log f) \dd \sigma,
\end{equation}
\begin{equation}
    \sum_{k=1}^3 \int_{S^2} f(b_k\cdot \grad \log f)^2 \dd \sigma =\int_{S^2} f |\grad_\sigma \log f|^2 \dd \sigma.
\end{equation}
It is not difficult to verify that
\begin{equation}
    D_{sph} \ge 2 \Lambda_3 R_{sph},
\end{equation}
after reducing the formulas to the Bakry-\'Emery $\Gamma_2$ criterion inequality \eqref{eq: Bakryemery} by applying the change of variables
\begin{equation}
    v=z+\rho \sigma, \quad w=z-\rho \sigma
\end{equation}
where $z=(v+w)/2 \in \R^3$ and $\rho \sigma=(v-w)/2$ for $\rho \in \R_{\ge 0}$ and $\sigma \in S^2$.
Therefore, from the formula \eqref{eq: decomposition}, it follows that
\begin{equation}
     -\frac12 \Jap{I'(F),Q(F)} \ge D_{sph}-\sup_{r>0} \left( \frac{r^2\alpha'(r)^2}{2\alpha(r)^2}\right) R_{sph} \ge 0
\end{equation}
if the interaction potential $\alpha$ satisfies
     \begin{equation} 
        \sup_{r>0}  \frac{r|\alpha'(r)|}{\alpha(r)} \le \sqrt{4\Lambda_3}.
    \end{equation}

In the proof for the monotonicity of the Landau equation in \cite{guillen2023landau}, the terms $D_{par}$ and $D_{rad}$ are neglected in the analysis, aside from their nonnegativity. However, for our proof of the dissipation estimate of the Fisher information, it is essential to exploit the term $D_{par}$. The term $D_{rad}$ remains unused, apart from its nonnegativity. 

\begin{lemma} \label{lem: newdissipation}
    \begin{equation}
        D_{par}+D_{sph} \ge \sum_{i,j=1,2,3} \iint_{\R^6} \frac{2\alpha}{|v-w|^2} F|b_i \cdot \grad_v(\tilde{b}_j \cdot \grad \log F)|^2 \dd w \dd v.
    \end{equation}
\end{lemma}
\begin{proof}
    Recall that $D_{par}$ is associated to the matrix
    \begin{equation}
        \frac12
        \begin{pmatrix}
            \begin{matrix} I_3 \\  \end{matrix} & I_3 \\ I_3 &  I_3 \\  
        \end{pmatrix},
    \end{equation}
    and $D_{sph}$ is associated to the matrix
    \begin{equation}
        \frac{1}{2|v-w|^2}
        \begin{pmatrix}
            \begin{matrix} a_{ij}(v-w) \\  \end{matrix} & -a_{ij}(v-w) \\ -a_{ij}(v-w) &  a_{ij}(v-w) \\ 
        \end{pmatrix}.
    \end{equation}
Observe that
   \begin{equation}
   \frac12
   \begin{pmatrix}
        \begin{matrix} I_3 \\  \end{matrix} & I_3 \\ I_3 &  I_3 \\  
    \end{pmatrix} 
    \ge \frac{1}{2|v-w|^2}
    \begin{pmatrix}
        \begin{matrix} a_{ij}(v-w) \\  \end{matrix} & a_{ij}(v-w) \\ a_{ij}(v-w) &  a_{ij}(v-w) \\ 
    \end{pmatrix},
\end{equation}
hence
\begin{equation}
    \frac12
   \begin{pmatrix}
        \begin{matrix} I_3 \\  \end{matrix} & I_3 \\ I_3 &  I_3 \\  
    \end{pmatrix} 
    +
    \frac{1}{2|v-w|^2}
        \begin{pmatrix}
            \begin{matrix} a_{ij}(v-w) \\  \end{matrix} & -a_{ij}(v-w) \\ -a_{ij}(v-w) &  a_{ij}(v-w) \\ 
        \end{pmatrix}
    \ge \frac{1}{|v-w|^2}
    \begin{pmatrix}
        \begin{matrix} a_{ij}(v-w) \\  \end{matrix} & 0 \\ 0 &  a_{ij}(v-w) \\ 
    \end{pmatrix}.
\end{equation}
It follows that
 \begin{equation} 
        D_{par}+D_{sph} \ge \sum_{i,j=1,2,3} \iint_{\R^6} \frac{\alpha}{|v-w|^2} F|b_i \cdot \grad_v(\tilde{b}_j \cdot \grad \log F)|^2 \dd w \dd v + \iint_{\R^6} \frac{\alpha}{|v-w|^2} F|b_i \cdot \grad_w(\tilde{b}_j \cdot \grad \log F)|^2 \dd w \dd v.
    \end{equation}
    By the symmetry, the two terms of the right hand side are equal.
\end{proof}
\begin{prop} \label{prop: dissipationnew}
    Under the same assumptions in Theorem \ref{thm: fisherdissipation}, the inequality
    \begin{equation}
        D_{par}+D_{sph} \ge c \int_{\R^3} \Jap{v}^{\gamma-2} f(v) \norm{\grad^2 \log f(v)}^2 \dd v -C_3 M_0 \cdot i(f) -C_4 M_0^2
    \end{equation}
    holds for some absolute constants $C_3, C_4$ and some $c>0$ only depending on $M_0, E_0, H_0$.
\end{prop}
\begin{proof}

Plugging in $F(v,w)=f(v)f(w)$ to Lemma \ref{lem: newdissipation}, we obtain
\begin{align}
    D_{par}+D_{sph}
    &\ge \sum_{i,j=1,2,3} \iint_{\R^6} \frac{2\alpha}{|v-w|^2} f(v)f(w) |b_i\cdot \grad_v(b_j\cdot(\grad_v-\grad_w) \log [f(v)f(w)] )|^2 \dd w \dd v.
\end{align}
We consider the cutoff $\tilde{\alpha}(r)$ of $\alpha(r)$ to remove the singularity at the origin. More precisely, we define
\begin{equation}
    \tilde{\alpha}(r):=\eta(r) \alpha(r)=\eta(r) r^\gamma
\end{equation}
where an increasing function $\eta(r) \in C^\infty(\R)$ equals to $1$ for $r\in [1,\infty)$, equals to $r^5$ for $r \in [0,\frac12]$, and $ \eta''(r)\le 2$ for every $r$. 
It is obvious that
\begin{equation} \label{eq: dissipationnew}
    D_{par}+D_{sph}
    \ge \sum_{i,j=1,2,3} \iint_{\R^6} \frac{2\tilde{\alpha}}{|v-w|^2} f(v)f(w) |b_i\cdot \grad_v(b_j\cdot(\grad_v \log f(v)-\grad_w \log f(w) )|^2 \dd w \dd v.
\end{equation}
Since
\begin{equation}
    b_i\cdot \grad_v (b_j \cdot (\grad_v \log f(v)-\grad_w \log f(w)))= b_i \cdot(\grad_v^2 \log f(v)) b_j +(b_i\cdot \grad_v b_j) (\grad_v \log f(v)-\grad_w\log f(w)),
\end{equation}
it follows that 
\eqref{eq: dissipationnew} is greater than or equal to
\begin{align}
  &\sum_{i,j=1,2,3} \iint_{\R^6} \frac{\tilde{\alpha}}{|v-w|^2} f(v)f(w)  \left(|b_i\cdot \grad_v^2 \log f(v) b_j|^2 -2 |(b_i\cdot \grad_v b_j) (\grad_v \log f(v)-\grad_w\log f(w))|^2 \right) \dd w \dd v
\end{align}
using $|x+y|^2 \ge \frac12 |x|^2-|y|^2$ for $x,y\in \R$.
Denote as
\begin{equation} \label{eq: defJ1}
    J_1:= \sum_{i,j=1,2,3} \iint_{\R^6} \frac{\tilde{\alpha}(|v-w|)}{|v-w|^2} f(v) f(w) |b_i\cdot \grad_v^2 \log f(v) b_j|^2 \dd w \dd v,
\end{equation}
\begin{equation} \label{eq: defJ2}
    J_2:=\sum_{i,j=1,2,3} \iint_{\R^6} \frac{2\tilde{\alpha}(|v-w|)}{|v-w|^2} f(v)f(w) |(b_i\cdot \grad_v b_j) (\grad_v \log f(v)-\grad_w\log f(w))|^2 \dd w \dd v,
\end{equation}
and we estimate $J_1$ and $J_2$ separately.
First, let's estimate $J_1$. Note that
\begin{align}
    \sum_{i,j=1,2,3}|b_i\cdot \grad_v^2 \log f(v) b_j|^2
    &=\sum_{j=1}^3 b_j\cdot \grad_v^2 \log f(v) a_{ij}(v-w) \grad_v^2 \log f(v) b_j\\
    &= \tr( \grad_v^2 \log f(v) a_{ij}(v-w) \grad_v^2 \log f(v) a_{ij}(v-w)).
\end{align}
We claim that
\begin{equation} \label{eq: newprop}
     \int_{\R^3} \frac{\tilde{\alpha}(|v-w|)}{|v-w|^2} f(w) \tr(a_{ij}(v-w) H a_{ij}(v-w) H)  \dd w \ge  c
     \Jap{v}^{\gamma-2}\norm{H}^2
\end{equation}
for every symmetric $3\times 3$ matrix $H$ where $c$ only depending on $M_0,E_0,H_0$. The idea of the proof is analogous to that of Lemma \ref{lem: lowerdiffusion}. Note that a version with a cutoff matrix for the Coulomb potential is given in \cite[Lemma 2.3]{ji2024entropy}. 

Applying the spectral decomposition for a symmetric matrix, it is enough to show that
\begin{equation} \label{eq: coercivitiy}
     \int_{\R^3} \frac{\tilde{\alpha}(|v-w|)}{|v-w|^2} f(w) \tr(a_{ij}(v-w) e e^t a_{ij}(v-w) e e^t)  \dd w \ge  c
     \Jap{v}^{\gamma-2}
\end{equation}
for a unit vector $e\in \R^3$. Note that
\begin{equation}
    \tr(a_{ij}(v-w) e e^t a_{ij}(v-w) e e^t) =\langle a_{ij}(v-w) e ,e \rangle^2=\left( |v-w|^2-|(v-w)\cdot e|^2 \right)^2,
\end{equation}
so \eqref{eq: coercivitiy} is equivalent to
\begin{equation} \label{eq: coercivitiy2}
     \int_{\R^3} \frac{\tilde{\alpha}(|v-w|)}{|v-w|^2}  
 \left( |v-w|^2-|(v-w)\cdot e|^2 \right)^2 f(w)   \dd w  \ge  c \Jap{v}^{\gamma-2}.
\end{equation}
The left hand side of \eqref{eq: coercivitiy2} is equal to
\begin{equation} \label{eq: coercivitiy3}
    \int_{\R^3} \frac{\tilde{\alpha}(|w|)}{|w|^2}  \left( |w|^2-|w\cdot e|^2 \right)^2 f(v-w) \dd w=\int_{\R^3} \eta(|w|) |w|^{\gamma+2}  \left( 1-\frac{|w\cdot e|^2}{|w|^2} \right)^2 f(v-w) \dd w
\end{equation}
By Lemma \ref{lem: concentration}, there exist $R,\mu,l$ such that
\begin{equation}
    |\{ w \in B_R(v) : f(v-w) \ge l \} | \ge \mu. 
\end{equation}
On the other hand, as in the proof of \cite[Lemma 3.2]{silvestre2017upperboundLandau},
\begin{equation}
    |\{ w \in B_R(v) :  1-\frac{|w\cdot e|^2}{|w|^2}<\eps_0 \}| \lesssim R (\sqrt{\eps_0}(|v|+R))^2 = \eps_0 R(|v|+R)^2.
\end{equation}
Choose $\eps_0 \approx \frac{\mu}{2R} (|v|+R)^{-2}$ so that
\begin{equation}
    |\{ w \in B_R(v) :  1-\frac{|w\cdot e|^2}{|w|^2}<\eps_0 \}| \le \frac{\mu}{2}.
\end{equation}
Then, either
\begin{equation}
    |\{ w \in B_R(v) \cap B_{1/2}(0) :  1-\frac{|w\cdot e|^2}{|w|^2} \ge \eps_0 \text{ and } f(v-w) \ge l \}|
\end{equation}
or
\begin{equation}
      |\{ w \in B_R(v) \cap B_{1/2}^c(0) :  1-\frac{|w\cdot e|^2}{|w|^2} \ge \eps_0  \text{ and } f(v-w) \ge l \}| 
\end{equation}
is greater than or equal to $\frac{\mu}{4}$. Depending on the cases, it follows that \eqref{eq: coercivitiy3} is greater than or equal to
\begin{equation}
    \min \left( \eps_0^2 l \int_{B_{r_0}} |w|^{\gamma+7} \dd w , \eps_0^2 l \frac{\mu}{4} \frac{(|v|+R)^{\gamma+2}}{32} \right) \approx \min \left( \eps_0^2 l \mu^{(\gamma+10)/3},\eps_0^2 l \mu (|v|+R)^{\gamma+2} \right)
\end{equation}
where $|B_{r_0}|=\frac{\mu}{4}$.
Hence,
\begin{equation}
     \int_{\R^3} \frac{\tilde{\alpha}(|v-w|)}{|v-w|^2} f(w) \tr(a_{ij}(v-w) e e^t a_{ij}(v-w) e e^t)  \dd w \gtrsim \min \left( \frac{ l \mu^{(\gamma+16)/3}}{R^2(|v|+R)^4}, \frac{l \mu^3 (|v|+R)^{\gamma-2}}{R^2} \right) \ge c \Jap{v}^{\gamma-2}
\end{equation}
for some $c>0$ depending on $R,\mu,l$, which proves \eqref{eq: newprop}. It follows that
\begin{equation}
    J_1 \ge c \int_{\R^3} \Jap{v}^{\gamma-2} f(v)\norm{\grad_v^2 \log f(v)}^2 \dd v.
\end{equation}
Let's estimate $J_2$.
Since
\begin{equation}
    \frac{1}{|v-w|^2}\sum_{i=1}^3 b_ib_i^t=\frac{1}{|v-w|^2} a(v-w) \le I_3,
\end{equation}
it follows that 
\begin{equation} \label{eq: estimateJ2}
    |J_2| \le \sum_{j=1}^3 \iint_{\R^6} 2\tilde{\alpha}(|v-w|)f(v)f(w) |(\grad_v b_j)(\grad_v \log f(v)-\grad_w \log f(w))|^2 \dd w \dd v
\end{equation}
Recall that
\begin{equation}
    b_j(v-w)=e_j\times(v-w) \Rightarrow \grad_v b_j(v-w)=e_j \times I_3.
\end{equation}
From \eqref{eq: estimateJ2}, we deduce
\begin{align}
    |J_2|
    & \le \sum_{j=1}^3 \iint_{\R^6} 2\tilde{\alpha}(|v-w|)f(v)f(w) |(e_j \times I_3)(\grad_v \log f(v)-\grad_w \log f(w))|^2 \dd w \dd v\\
    & = \sum_{j=1}^3 \iint_{\R^6} 2\tilde{\alpha}(|v-w|)f(v)f(w) |e_j \times (\grad_v \log f(v)-\grad_w \log f(w))|^2 \dd w \dd v\\
    & = 4 \iint_{\R^6} \tilde{\alpha}(|v-w|)f(v)f(w) | \grad_v \log f(v)-\grad_w \log f(w)|^2 \dd w \dd v.
\end{align}
Expand $|\grad_v \log f(v) -\grad_w \log f(w)|^2$ and use the symmetry to obtain 
\begin{equation}
    |J_2| \le 8 \iint_{\R^3} \left(\int_{\R^3} \tilde{\alpha}(|v-w|)f(w) \dd w \right) f(v) |\grad_v \log f(v)|^2 \dd v  -8\iint_{\R^6} \tilde{\alpha}(|v-w|) \grad_v f(v) \grad_w f(w) \dd w \dd v.
\end{equation}
From the definition of $\tilde{\alpha}$, it is obvious that
$\tilde{\alpha}(|v-w|) \le 2^{-\gamma}$. Hence,
\begin{equation}
    \int_{\R^3} \tilde{\alpha}(|v-w|) f(w) \dd w \le 2^{-\gamma} \int_{\R^3} f(w) \dd w =2^{-\gamma} M_0.
\end{equation}
On the other hand,
\begin{align}
    -8\iint_{\R^6} \tilde{\alpha}(|v-w|) \grad_v f(v) \grad_w f(w) \dd w \dd v
    &=-8\iint_{\R^6} \grad_v \cdot  \grad_w\tilde{\alpha}(|v-w|)  f(v) f(w) \dd w \dd v\\
    &=8\iint_{\R^6} \lap_v \tilde{\alpha}(|v-w|)  f(v) f(w) \dd w \dd v.
\end{align}
using integration by parts. Note that
\begin{align}
    \lap_v \tilde{\alpha}(|v-w|)
    &= \eta \lap \alpha + 2\grad \eta \cdot \grad \alpha + \alpha \lap \eta\\
    &= \gamma(\gamma+1)\eta|v-w|^{\gamma-2} +2(\gamma+1)\eta'|v-w|^{\gamma-1} +\eta''|v-w|^{\gamma}\\
    & \le 6\eta  |v-w|^{\gamma-2}+\eta''|v-w|^\gamma.
\end{align}
If $|v-w| >\frac12$, then $ 6\eta  |v-w|^{\gamma-2}+\eta''|v-w|^\gamma\le 6|v-w|^{\gamma-2}+2|v-w|^\gamma \le 26\cdot 2^{-\gamma}.$ Otherwise, if $|v-w| \le \frac12$, then
$6\eta  |v-w|^{\gamma-2}+\eta''|v-w|^\gamma= 26|v-w|^{\gamma+3} \le 26 \cdot 2^{-\gamma-3}$. Thus, $\lap_v \tilde{\alpha} \le 26 \cdot 2^{-\gamma}$. 
We deduce that
\begin{equation}
    |J_2| \le 2^{3-\gamma} M_0 \cdot i(f) +2^{3-\gamma} \cdot 26 M_0^2
\end{equation}
and it concludes the proof of the proposition.
\end{proof}

\begin{remark}
    The cutoff of $\alpha$ is necessary to bound $J_2$ from above. Without the cutoff, we cannot control $J_2$ using hydrodynamic bounds($M_0,E_0,H_0$) or higher $L^1$ moments.
\end{remark}   
\begin{proof}[Proof of Theorem \ref{thm: fisherdissipation}]
    Recall that
    \begin{equation}
        -\frac12 \partial_t I(F) =D_{par}+D_{rad}+D_{sph}-\iint F \frac{\alpha'^2}{2\alpha} |\tilde{b}_k\cdot \grad\log F|^2 \dd v \dd w.
    \end{equation}
    from \eqref{eq: decomposition}. Thus,
    \begin{align}
        -\frac12 \partial_t I(F)
        &\ge D_{par}+D_{sph} -\sup_{r>0} \left(\frac{r^2\alpha'(r)^2}{2\alpha(r)^2} \right) R_{sph}\\
         &\ge D_{par}+D_{sph} -\frac{1}{4\Lambda_3}\sup_{r>0} \left(\frac{r^2\alpha'(r)^2}{\alpha(r)^2} \right) D_{sph}\\
         &\ge \left[1- \frac{1}{4\Lambda_3}\sup_{r>0} \left(\frac{r^2\alpha'(r)^2}{\alpha(r)^2} \right)  \right] (D_{par}+D_{sph}).
    \end{align}
    Since we are assuming $\alpha(r)=r^\gamma$, we obtain
    \begin{equation}
        -\frac12 \partial_t I(F) \ge \left( 1-\frac{\gamma^2}{4\Lambda_3}\right) (D_{par}+D_{sph}).
    \end{equation}
    Lemma \ref{lem: liftedFisher}, Theorem \ref{thm: BakryEmery}, and Proposition \ref{prop: dissipationnew} implies that
    \begin{equation}
        - \partial_t i(f) \ge \left(1- \frac{9}{22 }\right) (D_{par}+D_{sph}) \ge  c_1 \int_{\R^3} \Jap{v}^{\gamma-2} f(v) \norm{\grad_v^2 \log f(v)}^2 \dd v -C_1 M_0 \cdot i(f) -C_2M_0^2.
    \end{equation}
\end{proof}

\section{Global existence of smooth solutions for initial data in $L^1_{2-\gamma} \cap L\log L$}
In this section, we prove the global existence of smooth solutions to the Landau equation \eqref{eq: Landaueq} for $\alpha(r)=r^\gamma$ with $\gamma \in [-3,-2)$, given initial data in $L^1_{2-\gamma} \cap L \log L$. 
\noindent
We follow the method outlined in \cite{golding2024local}, which concerns the case of the Coulomb potential.

\begin{proof} [Proof of Theorem \ref{thm: globalexistence}]
  
Consider a nonnegative initial data $f_0$ with mass $M_0$, energy $E_0$, and entropy $H_0$. Moreover, assume that 
\begin{equation}
    \int_{\R^3} f_0(v)|v|^{2-\gamma} \dd v \le W_0.
\end{equation}
\textbf{Step 1 : Approximation}\\
Approximate $f_0$ with a family of functions $\{f_0^\eps\}_{\eps>0}$ in $\R^3$ so that
\begin{enumerate}
    \item $f_0^\eps \in C^\infty_c(\R^3)$.
    \item $f_0^\eps$ has mass $M_0$, and energy $E_0$.
    \item $f_0^\eps$ satisfies the bounds
    \begin{equation}
        \int_{\R^3} f_0^\eps(v) |v|^{2-\gamma} \le 2W_0, \quad \int_{\R^3 } f_0^\eps(v) \log f_0^\eps(v) \dd v \le 2H_0.
    \end{equation}
    \item $f_0^\eps \to f_0$ strongly in $L^1_{2-\gamma}(\R^3)$ and pointwise a.e. as $\eps \to 0$.
\end{enumerate}
For each $\eps>0$, there exists a global in time Schwartz solution $f^\eps:[0,\infty)\times \R^3 \to \R_{\ge 0}$ with the initial data $f_0^\eps$ by Theorem \ref{thm: luisglobal}. Indeed, the Fisher information of $f^\eps_0$ is bounded since
\begin{equation}
    i(f_0^\eps) \le  C(\delta) \norm{f_0^\eps}_{H^2_{(3/2)+\delta}}
\end{equation}
holds for a given $\delta>0$, which was observed in \cite{toscani2000trend}.\\
\textbf{ Step 2: Uniform estimates}\\
For a given $T>0$, the bound 
\begin{equation}
     C_{Sobolev} \norm{f^\eps(t)}_{L^3(\R^3)} \le  i(f^\eps(t)) \le C_0(\gamma, T,M_0,W_0,H_0 )\left(1+\frac1t\right)
\end{equation}
holds for all $t\in (0,T]$ by Theorem \ref{thm: Fisherdecay} and the Sobolev embedding.
For small $\tau>0$, it follows that
\begin{equation} \label{eq: condition regularity}
    \norm{f^\eps}_{L^\infty([\tau/2,T]; L^3( \R^3)) }\le \frac{ C_0(\gamma, T,M_0,W_0,H_0 ) }{C_{Sobolev}} \left(1+\frac{2}{\tau}\right). 
\end{equation}
On the other hand, from Theorem \ref{thm: highermoment}, we have the estimate
\begin{equation} \label{eq: highermoment}
    \norm{f^\eps}_{L^\infty([\tau/2, T ];L^1_{2-\gamma}(\R^3))} \le C(\gamma,T,M_0,W_0, H_0).
\end{equation}
Interpolating \eqref{eq: condition regularity} and \eqref{eq: highermoment} yields
\begin{equation}
    \norm{f^\eps}_{L^\infty([\tau/2,T];L^p_k(\R^3))} \le C(\gamma, \tau, T, M_0, W_0, H_0).
\end{equation}
for $k=(2-\gamma)\frac{3-p}{2p}$. A simple computation show that there exists $p \in (\frac{3}{5+\gamma},\frac{3}{3+\gamma})$ such that
\begin{equation}
    (2-\gamma)\frac{3-p}{2p} > \frac1p \max\left(3p-3,-\frac{3\gamma}{2}-1 \right).
\end{equation}
For example, when $\gamma=-3$, we can take $p=1.55$.
Applying Theorem \ref{thm: conditional}, we obtain the uniform $L^\infty$ upper bound of $f$:
\begin{equation} \label{eq: Linftybound}
     \norm{f^\eps}_{L^\infty([\tau,T]\times \R^3) }\le C(\gamma,\tau,T,M_0,W_0,H_0)
\end{equation}
for some $C>0$.\\
\textbf{ Step 3 : Higher regularity estimates}\\
With the uniform $L^\infty $ bound of $f^\eps$, we can interpret the Landau equation \eqref{eq: Landaueq} as a uniform parabolic equation:
\begin{equation}
    \partial_t f^\eps =\grad \cdot (A[f^\eps] f^\eps -b [f^\eps] f^\eps ).
\end{equation}
Indeed, the diffusion matrix $A[f^\eps]$ is uniformly elliptic locally and $b[f^\eps]$ is uniformly bounded from above by \eqref{eq: upperdiffusion}, \eqref{eq: upperfirstorder}, and Lemma \ref{lem: lowerdiffusion}.
 Applying the standard technique of the parabolic De Giorgi-Nash-Moser and Schauder esitmates, we obtain  higher regularity estimates:
 For a given $R>0$ and a positive integer $m$, we have
\begin{equation} \label{eq: higherregularity}
    \norm{f^\eps}_{C^m([\tau,T]\times B_R)} \le C(m,R, \gamma,\tau,T,M_0,W_0,H_0).
\end{equation}
See \cite[Lemma 5.2]{golding2024local} for the detailed argument. \\
\textbf{ Step 4: Compactness and passing to the limit}\\
As a consequence of the uniform estimates in \eqref{eq: Linftybound} and \eqref{eq: higherregularity}, there exists a limit function $f:[0,T]\times \R^3 \to \R_{\ge 0}$ so that the convergence $f^\eps \to f$ holds in the following topologies:
\begin{enumerate}
    \item strongly in $C^m([\tau,T]\times B_R)$ for each positive integer $m$, for each $R>0$, and for each $0<\tau<T$
    \item pointwise a.e. in $[0,T]\times \R^3$
    \item strongly in $L^1([0,T]\times \R^3)$
    \item weak-$\ast$ly in $L^\infty([0,T];L^1_{2-\gamma}(\R^3))$  
\end{enumerate}
Note that the strong convergence in $L^1$ comes from uniform integrability, pointwise convergence, tightness. In particular, the limit function $f$ satisfies $f \ge 0$ and $f\in L^\infty([0,T];L^1_2(\R^3))$ by Fatou's lemma.
Note that if we show $f$ is a solution of the Landau equation, then $f$ automatically satisfies normalization. 
For $\varphi \in C_c^\infty((\tau,T]\times \R^3)$, we have
\begin{equation} \label{eq: test}
    \int_{\tau}^T \int_{\R^3} \partial_t f^\eps \varphi \dd v \dd t =-\int_{\tau}^T\int_{\R^3} \grad \varphi \cdot (A[f^\eps]\grad f^\eps- b[f^\eps] f^\eps ) \dd v \dd t.
\end{equation}
Since
\begin{equation}
    \norm{A[f^\eps-f]}_{L^1([\tau,T];L^\infty(\R^3))} \le \norm{f^\eps-f}_{L^1([\tau,T];L^1(\R^3))}^{1+\frac{\gamma+2}{3}} \norm{f^\eps-f}_{L^1([\tau,T];L^\infty(\R^3))}^{\frac{-\gamma-2}{3}}, 
\end{equation}
\begin{equation}
    \norm{b[f^\eps-f]}_{L^1([\tau,T];L^\infty(\R^3))} \le \norm{f^\eps-f}_{L^1([\tau,T];L^1(\R^3))}^{1+\frac{\gamma+1}{3}}
    \norm{f^\eps-f}_{L^1([\tau,T];L^\infty(\R^3))}^{\frac{-\gamma-1}{3}}, 
\end{equation}
we have the convergence of coefficient terms as $\eps \to 0$. Therefore, with the strong convergence in $C^1$, we obtain
\begin{equation}
    \int_\tau^T \int_{\R^3} \partial_t f \varphi \dd v \dd t =-\int_\tau^T \int_{\R^3} \grad \varphi \cdot (A[f]\grad f-b[f] f) \dd v \dd t.
\end{equation}
Thus, the limit $f$ solves the equation \eqref{eq: Landaueq} in a distributional sense, and it is smooth for any positive time. 
We conclude that $f$ solves the Landau equation \eqref{eq: Landaueq} in the classical sense for $(0,T]\times \R^3.$\\
\textbf{ Step 5: Behavior at the initial time}\\
For any test function $\psi \in C^2_c (\R^3)$, we have
\begin{equation}
    \int_{\R^3} f^\eps(\tau) \psi \dd v -\int_{\R^3}f^\eps_0 \psi \dd v =\int_0^\tau \psi \partial_t f^\eps \dd v \dd t
\end{equation}
for every $\tau \in [0,T]$, by integration by parts.
By symmetrizing and the Cauchy-Schwarz inequality,
\begin{align}
    &\int_0^\tau \int_{\R^3}  \psi \partial_t f^\eps \dd v \dd t\\
    =&-\frac12 \int_0^\tau \iint_{\R^6} \frac{a(v-w)}{|v-w|^2} (\grad_v -\grad_w)\left[f^\eps(v)f^\eps(w)|v-w|^{\gamma+2}\right] (\grad_v \psi (v)-\grad_w \psi(w)) \dd w \dd v \dd t\\
    =& \int_0^\tau \iint_{\R^6} \frac{a(v-w)}{|v-w|^2} (\grad_v -\grad_w)\left[\sqrt{f^\eps(v)f^\eps(w)|v-w|^{\gamma+2}}\right]\sqrt{f^\eps(v)f^\eps(w)|v-w|^{\gamma+2}} (\grad_v \psi(v)-\grad_w \psi(w)) \dd w \dd v \dd t\\
    \le & \left( \int_0^\tau \iint_{\R^6} X_\eps^2 \dd w \dd v \dd t\right)^{1/2} \left( \int_0^\tau \iint_{\R^6} Y_\eps^2 \dd w \dd v \dd t\right)^{1/2}
\end{align}
where
\begin{equation}
    X_\eps :=\frac{a(v-w)}{|v-w|^2} (\grad_v -\grad_w)\left[\sqrt{f^\eps(v)f^\eps(w)|v-w|^{\gamma+2}}\right]  ,
\end{equation}
\begin{equation}
    Y_\eps := \sqrt{f^\eps(v)f^\eps(w) |v-w|^{\gamma+2}} (\grad_v \psi(v)-\grad_w \psi(w))  .
\end{equation}
Note that
\begin{align}
    \int_0^\tau \iint_{\R^6} X_\eps^2 \dd w \dd v \dd t 
    &= \frac14 \int_0^\tau \iint_{\R^6} f^\eps(v)f^\eps(w)|v-w|^{\gamma} a(v-w) \left(\frac{\grad f^\eps(v)}{f^\eps(v)}-\frac{\grad f^\eps(w)}{f^\eps(w)}\right) \left(\frac{\grad f^\eps(v)}{f^\eps(v)}-\frac{\grad f^\eps(w)}{f^\eps(w)}\right)  \dd w \dd v \dd t\\
    &=\frac14 \int_0^\tau \partial_t h(f^\eps(t)) \dd t=\frac14 \left(h(f^\eps(0))-h(f^\eps(\tau)) \right) \le C(M_0,E_0,H_0).
\end{align}
On the other hand,
\begin{align}
    \int_0^\tau \iint_{\R^6} Y_\eps^2 \dd w \dd v \dd t 
    &\le \norm{\grad^2 \psi}_{L^\infty(\R^3)} \int_0^\tau  \iint_{\R^6} f^\eps(v)f^\eps(w)  |v-w|^{\gamma+4} \dd w \dd v \dd t\\
    &\le  C(\gamma) \norm{\grad^2 \psi}_{L^\infty(\R^3)} \int_0^\tau \iint_{\R^6} f^\eps(v)f^\eps(w) (1+|v-w|^2) \dd w  \dd v \dd t\\
    & \le C(\gamma,M_0,E_0) \tau \norm{\grad^2 \psi}_{L^\infty(\R^3)}.
\end{align}
We have shown that
\begin{equation}
    \left|\int_{\R^3} f^\eps(\tau) \psi \dd v -\int_{\R^3} f_0^\eps \psi \dd v \right|\le C(\gamma,M_0,E_0,H_0) \tau^{1/2} \norm{\grad^2 \psi}_{L^\infty(\R^3)}^{1/2}.
\end{equation}
By sending $\eps \to 0$, we obtain
\begin{equation}
     \left|\int_{\R^3} f(\tau) \psi \dd v -\int_{\R^3} f_0 \psi \dd v \right|\le C(\gamma,M_0,E_0,H_0) \tau^{1/2} \norm{\grad^2 \psi}_{L^\infty(\R^3)}^{1/2}.
\end{equation}
\end{proof}
\begin{remark}
    The estimates in Step $5$ are essentially coming from the weak $H$-formulation in \cite{villani1998newclass}, where the following notation for the projection matrix $\Pi$ is used:
\begin{equation}
    \Pi(v):=\frac{a(v)}{|v|^2}.
\end{equation}
It is easy to see that $\Pi^2= \Pi$ and $\Pi(\grad_v -\grad_w) |v-w|^{(\gamma+2)/2}$.
\end{remark}

\appendix
\section{A proof of Remark \ref{remark: LlogL} } \label{appendix : 1}
Suppose that a positive function $f$ on $\R^3$ satisfies
\begin{equation}
    \int_{\R^3} f(v) \dd v= M_0, \quad \int_{\R^3 } f(v) |v|^2 \dd v =E_0, \quad \int_{\R^3 }f(v) \log f(v)\le H_0.
\end{equation}
Then,
\begin{align}
    \int_{\R^3} f(v) |\log f(v)| \dd v
    &= -2 \int_{ \{ 0< f < 1\} } f(v) \log f(v) \dd v +\int_{ \R^3 } f(v) \log f(v) \dd v \\
    & \le  -2 \int_{ \{ f \le e^{-1-|v|^2} \} } f(v) \log f(v) \dd v -2 \int_{ \{ e^{-1-|v|^2} \le f < 1\} } f(v) \log f(v) \dd v + H_0\\
    & \le  -2 \int_{ \{ f \le e^{-1-|v|^2} \} } f(v) \log f(v) \dd v +2 \int_{ \{ e^{-1-|v|^2} \le f < 1\} } f(v) \left(1+|v|^2 \right) \dd v + H_0\\
    & \le  -2 \int_{ \{ f \le e^{-1-|v|^2} \} } f(v) \log f(v) \dd v +2M_0 +2E_0 + H_0.
\end{align}
Noticing that the function $ -x \log x$ is increasing for $x\in (0,\frac1e]$, it follows that
\begin{equation}
     -2 \int_{ \{ f \le e^{-1-|v|^2} \} } f(v) \log f(v) \dd v \le 2 \int_{ \{ f \le e^{-1-|v|^2} \} } e^{-1-|v|^2} \left(1+|v|^2 \right)  \dd v \le 2 \int_{\R^3} e^{-1-|v|^2} \left(1+|v|^2 \right)  \dd v.
\end{equation}
Since $e^{-1-|v|^2} \left(1+|v|^2 \right)$ is integrable on $\R^3$, Remark \ref{remark: LlogL} is proved.

\bibliographystyle{abbrv}
\bibliography{citeFisher}

\begin{thebibliography}{10}

\bibitem{alexandre2000entropydissipation}
R.~Alexandre, L.~Desvillettes, C.~Villani, and B.~Wennberg.
\newblock Entropy dissipation and long-range interactions.
\newblock {\em Arch. Ration. Mech. Anal.}, 152(4):327--355, 2000.

\bibitem{alonso2024prodiserrin}
R.~Alonso, V.~Bagland, L.~Desvillettes, and B.~Lods.
\newblock A~priori estimates for solutions to {L}andau equation under
  {P}rodi-{S}errin like criteria.
\newblock {\em Arch. Ration. Mech. Anal.}, 248(3):Paper No. 42, 63, 2024.

\bibitem{bakry1985diffusion}
D.~Bakry and M.~\'Emery.
\newblock Diffusions hypercontractives.
\newblock In {\em S\'eminaire de probabilit\'es, {XIX}, 1983/84}, volume 1123
  of {\em Lecture Notes in Math.}, pages 177--206. Springer, Berlin, 1985.

\bibitem{bakry2014book}
D.~Bakry, I.~Gentil, and M.~Ledoux.
\newblock {\em Analysis and geometry of {M}arkov diffusion operators}, volume
  348 of {\em Grundlehren der mathematischen Wissenschaften [Fundamental
  Principles of Mathematical Sciences]}.
\newblock Springer, Cham, 2014.

\bibitem{bowman2024isotropiclandau}
D.~Bowman and S.~Ji.
\newblock Global existence for an isotropic landau model, 2024.

\bibitem{carrapatoso2017large}
K.~Carrapatoso, L.~Desvillettes, and L.~He.
\newblock Estimates for the large time behavior of the {L}andau equation in the
  {C}oulomb case.
\newblock {\em Arch. Ration. Mech. Anal.}, 224(2):381--420, 2017.

\bibitem{chern2022uniqueness}
J.-L. Chern and M.~Gualdani.
\newblock Uniqueness of higher integrable solution to the {L}andau equation
  with {C}oulomb interactions.
\newblock {\em Math. Res. Lett.}, 29(4):945--960, 2022.

\bibitem{desvillettes2015entropy}
L.~Desvillettes.
\newblock Entropy dissipation estimates for the {L}andau equation in the
  {C}oulomb case and applications.
\newblock {\em J. Funct. Anal.}, 269(5):1359--1403, 2015.

\bibitem{fournier2010uniqueness}
N.~Fournier.
\newblock Uniqueness of bounded solutions for the homogeneous {L}andau equation
  with a {C}oulomb potential.
\newblock {\em Comm. Math. Phys.}, 299(3):765--782, 2010.

\bibitem{fournier2009local}
N.~Fournier and H.~Gu\'erin.
\newblock Well-posedness of the spatially homogeneous {L}andau equation for
  soft potentials.
\newblock {\em J. Funct. Anal.}, 256(8):2542--2560, 2009.

\bibitem{golding2024globalsmoothsolutionslandaucoulomb}
W.~Golding, M.~Gualdani, and A.~Loher.
\newblock Global smooth solutions to the landau-coulomb equation in
  {$L^{3/2}$}, 2024.

\bibitem{golding2024local}
W.~Golding and A.~Loher.
\newblock Local-in-time strong solutions of the homogeneous {L}andau-{C}oulomb
  equation with {$L^p$} initial datum.
\newblock {\em Matematica}, 3(1):337--369, 2024.

\bibitem{gualdani2019apweights}
M.~Gualdani and N.~Guillen.
\newblock On {$A_p$} weights and the {L}andau equation.
\newblock {\em Calc. Var. Partial Differential Equations}, 58(1):Paper No. 17,
  55, 2019.

\bibitem{guillen2023landau}
N.~Guillen and L.~Silvestre.
\newblock The landau equation does not blow up, 2023.

\bibitem{imbert2024monotonicityfisherinformationboltzmann}
C.~Imbert, L.~Silvestre, and C.~Villani.
\newblock On the monotonicity of the fisher information for the boltzmann
  equation, 2024.

\bibitem{ji2024bakryemery}
S.~Ji.
\newblock Bounds for the optimal constant of the bakry-\'emery {$\Gamma_2$}
  criterion inequality on {$\R P^{d-1}$}, 2024.

\bibitem{ji2024entropy}
S.~Ji.
\newblock Entropy dissipation estimates for the landau equation with coulomb
  potentials, 2024.

\bibitem{mouhot2006coercivity}
C.~Mouhot.
\newblock Explicit coercivity estimates for the linearized {B}oltzmann and
  {L}andau operators.
\newblock {\em Comm. Partial Differential Equations}, 31(7-9):1321--1348, 2006.

\bibitem{silvestre2016newregularization}
L.~Silvestre.
\newblock A new regularization mechanism for the {B}oltzmann equation without
  cut-off.
\newblock {\em Comm. Math. Phys.}, 348(1):69--100, 2016.

\bibitem{silvestre2017upperboundLandau}
L.~Silvestre.
\newblock Upper bounds for parabolic equations and the {L}andau equation.
\newblock {\em J. Differential Equations}, 262(3):3034--3055, 2017.

\bibitem{toscani2000trend}
G.~Toscani and C.~Villani.
\newblock On the trend to equilibrium for some dissipative systems with slowly
  increasing a priori bounds.
\newblock {\em J. Statist. Phys.}, 98(5-6):1279--1309, 2000.

\bibitem{villani1998newclass}
C.~Villani.
\newblock On a new class of weak solutions to the spatially homogeneous
  {B}oltzmann and {L}andau equations.
\newblock {\em Arch. Rational Mech. Anal.}, 143(3):273--307, 1998.

\bibitem{wu2014soft}
K.-C. Wu.
\newblock Global in time estimates for the spatially homogeneous {L}andau
  equation with soft potentials.
\newblock {\em J. Funct. Anal.}, 266(5):3134--3155, 2014.

\end{thebibliography}

\end{document}